\documentclass{amsart}

\usepackage[utf8]{inputenc}
\usepackage{amsthm}
\usepackage[mathscr]{euscript}
\usepackage{tikz-cd}
\usepackage{url}
\usepackage{hyperref}

\newtheorem{theorem}{Theorem}
\numberwithin{theorem}{section}
\newtheorem{proposition}[theorem]{Proposition}

\newtheorem{corollary}[theorem]{Corollary}
\theoremstyle{definition} 
\newtheorem{definition}[theorem]{Definition}
\theoremstyle{remark}
\newtheorem{example}[theorem]{Example}
\newtheorem{setup}[theorem]{Setup}
\newtheorem{remark}[theorem]{Remark}
\numberwithin{theorem}{section}
\numberwithin{equation}{section}
\providecommand{\C}{\mathbb C} 
\providecommand{\FF}{\mathbb F}
\providecommand{\Frac}{\operatorname{Frac}}
\providecommand{\charac}{\operatorname{char}}
\providecommand{\Spec}{\operatorname{Spec}}
\providecommand{\Frob}{\mathrm{F}}
\providecommand{\A}{\mathbb A}

\providecommand{\PP}{\mathbb P}
\providecommand{\pr}{\mathrm{pr}}
\providecommand{\OO}{\mathscr O} 
\providecommand{\T}{\mathrm{T}}

\providecommand{\U}{\mathscr U}
\providecommand{\Hom}{\operatorname{Hom}}
\providecommand{\Sym}{\operatorname{Sym}}

\providecommand{\tors}{\mathrm{tors}}

\providecommand{\irr}{\operatorname{irr}}
\providecommand{\nr}{\operatorname{nr}}
\providecommand{\snr}{\operatorname{snr}}
\providecommand{\codim}{\operatorname{codim}}
\providecommand{\displayroot}[2]{\sqrt[\leftroot{-1}\uproot{2}{#1}]{#2}}
\providecommand{\pthroot}[1]{\displayroot p {#1}}

\title{On the degree of irrationality of complete intersections}
\author{Lucas Braune and Taro Yoshino}
\date{\today}
\address{Munich, Germany}
\email{lucasbraune@gmail.com}

\address{Graduate School of Mathematical Sciences, The University of Tokyo, 3-8-1 Komaba,
Meguro-ku, Tokyo, 153-8914, Japan}
\email{yotaro@ms.u-tokyo.ac.jp}

\begin{document}

\begin{abstract}
  Let $X\subseteq \PP^N$ be a very general complete intersection of hypersurfaces of degrees $d_1,\dotsc,d_c$ over the complex numbers.
  Let $p$ be a prime number.
  If $\sum d_i \ge \frac p {p+1} (N+2p-2) + c(p-1)$, then every generically finite, dominant map from $X$ to a rational variety has degree at least $p$.
\end{abstract}

\maketitle

\section{Introduction}

In this paper, we study the degree of irrationality of complete intersections  of hypersurfaces in projective space. By definition, the \emph{degree of irrationality} of a variety $X$ is the minimum degree of a dominant, generically finite rational map from $X$ to a rational variety. This invariant gives a measure of how far $X$ is from being rational. We denote it by $\irr(X)$.

If $X\subseteq \PP^N_\C$ is a smooth complete intersection of multi-degree $(d_1,\dotsc,d_c)$ and $\sum d_i \geq N + 1$, then $X$ is not rational, because its canonical sheaf carries a nonzero global section.
By a result of Bastianelli \emph{et al.}\ \cite[Theorem 1.10]{BdPELU17}, the degree of irrationality of $X$ satisfies 
\begin{equation*}
  \irr(X) \ge \sum_{i=1}^c d_i - N + 1.
\end{equation*}

It is harder to prove that complete intersections with $\sum d_i \le N$ are irrational.
In 1995, Kollár \cite{kollar1995, Kollar1996} made a breakthrough by showing that a very general hypersurface of degree $d$ in $\PP^N_\C$ is not ruled, and therefore not rational, provided that
\begin{equation*}
  d\ge 2\left\lceil \dfrac{N+2}3 \right\rceil
\end{equation*}
and $N\ge 4$.
A key feature of Kollár's argument is a degeneration to positive characteristic.

Chen and Stapleton \cite{CS20} used Kollár's argument to establish lower bounds for the degree of irrationality of very general hypersurfaces.
They proved that, if $X\subseteq \PP^N_\C$ is a very general hypersurface of degree $d\ge N - \sqrt{N+1}/4$ and dimension $\ge 3$, then
\begin{equation*}
  \irr(X) \ge \sqrt{N+1} / 4.
\end{equation*}

The first author \cite{Braune19b} extended Kollár's result to complete intersections, proving that a complete intersection of $c$ very general hypersurfaces of degrees $d_1,\dotsc, d_c$ in $\PP^N_\C$ it not ruled, and therefore not rational, provided that
\begin{equation*}
  \sum_{i=1}^c d_i \ge \dfrac 2 3 N + c + 1.
\end{equation*}

In this paper, we combine the techniques of \cite{CS20,Braune19b} to establish the following result.

\begin{theorem}
  \label{thm:invariants-intro}
  Let $X\subset \PP^N_\C$ be a very general complete intersection of positive dimension and multi-degree $(d_1,\dotsc, d_c)$ over $\C$.
  Let $p$ be a prime number such that $d_1,\dotsc, d_c \ge p$.
  If
  \begin{equation*}
    \sum_{i=1}^c d_i \ge \frac p {p+1} (N+2p-2) + c(p-1),
  \end{equation*}
  then every generically finite, dominant rational map from $X$ to a ruled variety has degree $\ge p$.
  In particular, $\irr(X)\ge p$.
\end{theorem}

The meanings of ``very general'', ``ruled'' and other standard terms used in this introduction are recalled in Subsection \ref{subsec:notation} of this paper.
We derive Theorem \ref{thm:invariants-intro} from Theorem \ref{thm:nonruledness-precise} in Section \ref{sec:proof of the main theorem}.
The reader seeking to compare Theorems \ref{thm:invariants-intro} and \ref{thm:nonruledness-precise} will note that the latter is stated in terms of the invariant $\nr(X)$, which is introduced in Definition \ref{def:invariants} below.
The conclusion of Theorem \ref{thm:invariants-intro} is equivalent to the inequality $\nr(X) \ge p$.

\subsection{Background}

The main result of Bastianelli \emph{et al.}\ \cite{BdPELU17} states that, if $X \subseteq \PP^N_\C$ is a very general hypersurface of degree $d \ge 2N-1$, then $\irr(X) = d - 1$.

Totaro \cite{Totaro2016} reused Koll\'ar's argument to prove that a very general hypersurface of degree $d$ in $\PP^N_\C$ is not stably rational under the slightly weaker hypotheses that $d\ge 2 \lceil (N+1)/3 \rceil$ and $N\ge 4$;
by definition, a variety $X$ is \emph{stably rational} if there exists $m\ge 0$ such that $X\times \PP^m$ is rational.

Cheltsov and Wotzlaw \cite{CW04} extended Kollár's argument to complete intersections, proving that a very general complete intersection of multi-degree $(d_1,\dotsc, d_c)$ in $\PP^N_\C$ is not ruled, provided that
\begin{equation*}
  d_1 \ge 2 \left\lfloor \dfrac{N+2 - \sum_{i\ge 2}d_i}{3} \right\rfloor.
\end{equation*}
and $N-c \ge 3$.
Chatzistamatiou and Levine \cite{CL16} extended Totaro's argument in similar fashion, proving a stable irrationality result for complete intersections satisfying similar hypotheses.

Today, the best irrationality result for hypersurfaces is due to Schreieder \cite{Schreieder2019}, which Nicaise--Otem \cite{NO20} generalised to complete intersections in the following form: a very general complete intersection $X \subseteq \PP^N_\C$ with 
\begin{equation*}
  d_1 \ge \max\left(
  \log_2\left( N - 1 - \textstyle\sum_{i\ge 2}d_i \right)+2,
  d_2,\dotsc, d_c, 4
  \right)
\end{equation*}
and $N - \sum_{i\ge 2} d_i\ge 2$ is not stably rational, and therefore not rational.

\subsection{Notation and Terminology}

\label{subsec:notation}

Let $k$ be a field. A \textit{complete intersection} in a projective space $\PP^N_k$ is a closed subscheme $X\subseteq\PP^N_k$ that can be defined by the vanishing of $c$ homogenous polynomials in the coordinates of $\PP^N_k$, where $c$ is the codimension of $X$ in $\PP^N_k$.

A variety is an integral scheme of finite type over an algebraically closed field.

\begin{remark}
  A smooth complete intersection of positive dimension over an algebraically closed field is a variety. This follows from the fact that every positive dimensional complete intersection is connected.
\end{remark}
A property holds for a \textit{very general} point of a variety $X$ if it holds for all closed points of X away from the union of countably many proper subvarieties.

A variety is \textit{ruled} if it is birationally equivalent to the product of another variety with the projective line.

The \emph{torsion} of a sheaf of modules $F$ on a variety $X$ is the kernel $F_\tors$ of the natural map $F\to F\otimes_{\OO_X} k(X)$, where $k(X)$ denotes the function field of $X$.

Given a variety $X$, we denote
by $\Omega_X$ its sheaf of Kähler differentials;
by $\Omega_X^i$ its sheaf of degree-$i$ differential forms $\wedge^i \Omega_X$.

The following table enumerates further notions and symbols that we will define near their first use in this paper.

\medskip

\begin{center}
  \begin{tabular}{|l|l|l|}
    \hline Notion                    & Symbol                & Definition                 \\ \hline
    Degree of nonruledness           & $\nr(X)$              & \ref{def:invariants}
    \\
    Degree of separable nonruledness & $\snr(X)$             & \ref{def:snr}
    \\
    Separation of general points     &                       & \ref{def:separates-pts}
    \\
    Universal domain                 & \                     & \ref{def:universal-domain}
    \\
    Cyclic cover                     & $X[\displayroot m s]$ & \ref{def:cyclic-covers}
    \\
    \hline
  \end{tabular}
\end{center}

\medskip
\subsection{Acknowledgement}
Part of this work was carried out while the first author was a Postdoctoral Scholar at the University of Washington. We thank Sándor Kovács for valuable discussions.
We are also grateful to Nathan Chen and David Stapleton for giving us helpful comments and suggestions.
The second author thanks his supervisor Yoshinori Gongyo and his senior Tasuki Kinjo for their encouragement.

\section{Obstructions to ruledness}

In this section, we establish minor extensions of results due to Chen--Stapleton and Matsusaka.
These extensions concern the following birational invariant, which measures how far a variety is from being ruled.

\begin{definition}
  \label{def:invariants}
  Let $X$ be a variety. The \emph{degree of nonruledness} $\nr(X)$ is the smallest natural number $d$ such that there exists a dominant rational map of degree $d$ from $X$ to a ruled variety.
\end{definition}

As far as we know, the degree of nonruledness did not appear explicitly in the literature prior to this paper.
Its definition is analogous to that of the degree of irrationality, which Chen and Stapleton \cite{CS20} used to phrase their results.
Denoting the degree of irrationality of a variety $X$ by $\irr(X)$, we trivially have
\begin{equation*}
  \mathrm{nr}(X) \le \mathrm{irr}(X)
\end{equation*}
for all $X$.
This inequality is strict, for example, if $X$ is ruled, but not rational.

Roughly speaking, the results in this section say
\begin{enumerate}
  \item that lower bounds for $\nr(X)$ can be derived from the existence of global differential forms on $X$ (Propositions \ref{prop:snr-top-forms} and \ref{prop:obstruction-to-ruledness}); and
  \item this invariant is lower semi-continuous, that is, it cannot ``jump up'' in families (Theorem \ref{thm:gen-matsusaka}).
\end{enumerate}

\subsection{Differential forms}
\label{subsec:forms}

A generically finite, dominant map between varieties $f : X\to Y$ is separable if, and only if, its differential $df : \T_X \to f^* \T_Y$ is an isomorphism at the generic point of $X$.
In this section, we use leverage this observation to obtain lower bounds for the following separable analog of the degree of nonruledness.

\begin{definition}
  \label{def:snr}
  Let $X$ be a variety.
  The \emph{separable degree of nonruledness} $\snr(X)$ is the smallest natural number $d$ such that there exists a separable dominant rational map of degree $d$ from $X$ to a ruled variety.
\end{definition}

\begin{remark}
  \label{rmk:deg-insep}
  The degree of an inseparable map is divisible by the characteristic.
  Therefore, if $X$ is a variety over a field of positive characteristic $p$, then
  \begin{equation*}
    \nr(X) \ge \min(\snr(X), p).
  \end{equation*}
\end{remark}

In Propositions \ref{prop:snr-top-forms} and \ref{prop:obstruction-to-ruledness} we assume that a variety $X$ has ``sufficiently many'' global differential forms.
To formulate this precisely, we make the following definition.

\begin{definition}
  \label{def:separates-pts}
  Let $X$ be a variety, let $M$ be an $\OO_X$-module, let $W \subseteq \Gamma(X, M)$ be a linear subspace, and let $n$ be a positive integer.
  We say that $W$ \emph{separates $n$ general points} if there exists a dense open subset $U\subseteq X$ with the following property: given distinct closed points $x_1,\dotsc, x_n\in U$, there exists a section $s\in W$ such that $s(x_i)=0\in M(x_i)$ for all $i=1,\dotsc, n-1$, but $s(x_n)\ne 0 \in M(x_n)$.
\end{definition}

\begin{remark}
  \label{rmk:sep-pts}
  Let $L$ be an invertible sheaf on a variety $X$.
  Let $W\subseteq \Gamma(X, L)$ be a linear subspace.
  \begin{itemize}
    \item If $W\ne 0$, then the rational map $X\dashrightarrow \PP W$ induced by $L$ is generically injective if, and only if, $W$ separates $2$ general points.
    \item $L$ is big if, and only if, there exists $m>0$ such that $\Gamma(X, L^{\otimes m})$ separates $2$ general points.
    \item If $L$ is very ample, then $\Gamma(X, L^{\otimes m})$ separates $m+1$ general points for each $m \ge 0$.
  \end{itemize}
\end{remark}

\begin{proposition}
  \label{prop:snr-top-forms}
  Let $X$ be a smooth proper variety.
  Let $n$ be a positive integer.
  If $\Gamma(X, \omega_X)$ separates $n$ general points, then $\snr(X) > n$.
\end{proposition}

A similar result holds for the degree of irrationality in characteristic zero; see \cite[Theorem 1.10]{BdPELU17}.

\begin{proof}
  We mimic the argument of \cite[Lemma 2.3]{CS20}.
  Aiming for a contradiction, suppose that $\Gamma(X, \omega_X)$ separates $n$ general points, and that there exists a variety $Y$ and a separable dominant rational map $f : X\dashrightarrow Y\times \PP^1$ of degree $\le n$.
  After shrinking $Y$, we may assume that $Y$ is smooth.
  Let $R \subseteq X\times Y\times \PP^1$ be the closure of the graph of $f$.
  We have a commutative diagram
  \begin{equation*}
    \begin{tikzcd}
      R \ar[rd, "p"] \ar[d, "q"'] \\
      X \ar[r, dashed, "f"'] & Y \times \PP^1
    \end{tikzcd}
  \end{equation*}
  where $q$ is birational and $p$ is proper, separable and of degree $\le n$.

  Let $N := \dim X$ and let
  \begin{equation*}
    \operatorname{Tr}_p^N : p_* \Omega_R^N \to \Omega_{Y\times\PP^1}^N,
  \end{equation*}
  be the $\OO_{Y\times\PP^1}$-linear \emph{trace map} of \cite[Proposition 3.3]{dJS2004} and \cite{dJS2008}.
  Let $V\subseteq Y\times \PP^1$ be an open subset over which $p$ is finite and \'etale.
  Over $V$, the trace map $\operatorname{Tr}_p^N$ is given by a sum over fibers.
  Let $U\subseteq X$ be an open subset as in Definition \ref{def:separates-pts} applied to $\Gamma(X, \omega_X)$ and over which $q$ is an isomorphism.
  After shrinking $V$, we may assume $p^{-1}V \subseteq q^{-1}U$.

  Let $v\in V$ be a closed point.
  The fiber $p^{-1}\{v\} \subseteq R$ consists of points $z_1,\dotsc,z_d$, whose images under $q$ are distinct and contained in $U$.
  Because $\deg(p) \le n$, there exists a section $s\in \Gamma(X, \omega_X)$ that vanishes at $q(z_1), \dotsc, q(z_{d-1})$, but not at $q(z_d)$.
  For such a section,
  \begin{equation*}
    \operatorname{Tr}_p^N(q^* s) \in \Gamma(Y\times \PP^1, \omega_{Y\times \PP^1})
  \end{equation*}
  is nonzero at $v$.
  This contradicts the fact that
  \begin{equation*}
    \Gamma(Y\times \PP^1, \omega_{Y\times \PP^1}) =
    \Gamma(Y, \omega_Y)\otimes_k \Gamma(\PP^1_k, \omega_{\PP^1}) = 0,
  \end{equation*}
  as follows from the Künneth formula and the isomorphism $\omega_{\PP^1} \cong \OO_{\PP^1}(-2)$.
\end{proof}

A direct application of Proposition \ref{prop:snr-top-forms} to complete intersections yields the following:

\begin{corollary}
  \label{cor:no-degeneration}
  Let $X\subseteq \PP^N$ be a smooth complete intersection of positive dimension and multi-degree $(d_1,\dotsc, d_c)$ over an algebraically closed field, and let
  \begin{equation*}
    \delta := \sum_{i=1}^c d_i - N - 1.
  \end{equation*}
  If $\delta \ge 0$, then $\snr(X) > \delta+1$.
\end{corollary}

\begin{proof}
  The canonical sheaf $\omega_X$ is naturally isomorphic to $\OO_X(\delta)$.
  If $\delta\ge 0$, then $\Gamma(X, \omega_X)$ separates $\delta+1$ general points (Remark \ref{rmk:sep-pts}).
  Therefore, the result follows from Proposition \ref{prop:snr-top-forms}.
\end{proof}

\begin{remark}
  \label{rmk:classical-irrationality}
  If $\delta \ge 0$, then $\Gamma(X, \omega_X) \neq 0$, so $X$ is not separably uniruled.
\end{remark}

The next proposition is a minor extension of a result of Chen--Stapleton \cite[Lemma 2.3]{CS20}.

\begin{proposition}[Chen--Stapleton]
  \label{prop:obstruction-to-ruledness}
  Let $X$ be a proper variety.
  Let $Q$ be an invertible sheaf on $X$.
  Suppose that, for some integer $i>0$, there exists an injection of $\OO_X$-modules
  \begin{equation*}
    Q \hookrightarrow \Omega_X^i/(\Omega_X^i)_\tors.
  \end{equation*}
  Let $n$ be a positive integer.
  If $\Gamma(X, Q)$ separates $n$ general points, then $\snr(X) > \lfloor n/2\rfloor$.
\end{proposition}

\begin{proof}
  The result reduces to \cite[Lemma 2.3]{CS20} when $\Omega_X^i$ is torsion-free.
  The proof given in loc.\ cit.\ works with minor modifications in the case where $\Omega_X^i$ has torsion, because for every separable, generically finite, dominant, proper morphism of varieties $p : R\to Y\times \PP^1$ with $Y$ smooth, the trace map
  \begin{equation*}
    \operatorname{Tr}_p^i : p_* \Omega_R^i \to \Omega_{Y\times\PP^1}^i
  \end{equation*}
  of \cite[Proposition 3.3]{dJS2004} factors through $p_* (\Omega_R^i/(\Omega_R^i)_\tors)$.
\end{proof}

\subsection{Matsusaka's theorem}

\label{subsec:matsusaka}

Matsusaka's theorem says that, given a family of varieties over a discrete valuation ring, if the generic fiber is ruled, then so is the special fiber \cite[p. 233]{Matsusaka68}.
Chen--Stapleton extended Matsusaka's theorem by showing that the degree of nonruledness is lower semi-continuous \cite[Theorem 1.1]{CS20}.

\begin{theorem}[Matsusaka, Chen--Stapleton]
  \label{thm:gen-matsusaka}
  Let $f : X\to S$ be a flat, finite-type morphism between excellent schemes.
  Let $\eta, s\in S$ be points such that $s$ is in the closure of $\eta$.
  Choose algebraic closures $\overline{k(\eta)}$ and $\overline{k(s)}$ of the residue fields of $\eta$ and $s$, and let $X_{\bar \eta}$ and $X_{\bar s}$ be the corresponding geometric fibers.
  Suppose that $X_{\bar \eta}$ is integral, and let $X'_{\bar s}\subseteq X_{\bar s}$ be an irreducible component at the generic point of which $X_{\bar s}$ is reduced.
  \begin{enumerate}
    \item If $X_{\bar \eta}$ is ruled, then so is $X_{\bar s}'$.
    \item If $X_{\bar \eta}$ admits a dominant rational map of degree $\le d$ to a ruled variety, then the same holds for $X'_{\bar s}$.
          In other words,
          \begin{equation*}
            \nr(X'_{\bar s})\le \nr(X_{\bar \eta}).
          \end{equation*}
  \end{enumerate}
\end{theorem}

\begin{proof}
  Part (1) follows from part (2).
  The result of Chen--Stapleton \cite[Theorem 1.1]{CS20} establishes part (2) in the case where $S$ is the spectrum of an excellent discrete valuation ring and $X$ is normal and irreducible.
  Indeed, \cite[Theorem 1.1]{CS20} holds
  for nonprojective finite-type morphisms, with the same proof.
  Let us reduce part (2) to the case considered by Chen--Stapleton.

  Let $S' \subseteq S$ be the closure of $\eta$ equipped with the reduced scheme structure, and let $R$ be an excellent DVR with fraction field $k(\eta)$ that dominates the local ring $\OO_{S', s}$.
  For example, letting $B$ denote the normalization of the blowup of $S'$ along the closure of $s$, we can take $R$ to be the local ring of any codimension $1$ point of $B$ that lies over $s$.

  Let $K$ be an algebraic closure of the residue field of $R$.
  By Remark \ref{rmk:ruled-geometric} below applied to $X'_{\bar s}$ and a $k(s)$-algebra embedding of $\overline{k(s)}$ into $K$, it suffices to prove the result for the second projection $X\times_S \Spec R\to \Spec R$.
  Thus we may assume that $S$ is the spectrum of an excellent DVR.
  We may further assume that $\eta\ne s$, so that $\eta$ (resp. $s$) is the generic (resp. closed) point of $S$.

  By flatness, the special fiber $X_s$ is a Cartier divisor in $X$.
  Its complement, the generic fiber $X_\eta$, is irreducible, so $X$ is irreducible.
  Let $x' \in X$ be the image of the generic point of $X'_{\bar s}$.
  The Cartier divisor $X_s\subseteq X$ is reduced at $x'$, so $X$ is normal at $x'$.
  Shrinking $X$ to an open neighborhood of $x'$, we may further assume that $X$ is normal.
  Doing so, we reach the case considered by Chen--Stapleton.
\end{proof}

\begin{remark}
  \label{rmk:ruled-geometric}
  Let $k\subseteq K$ be an extension of algebraically closed fields, and let $X$ be a variety over $k$.
  By \cite[Lemma 1.2]{CS20} and its proof, $X$ admits a rational map of degree $d$ to a ruled $k$-variety if, and only if, $X\times_k K$ admits a rational map of degree $d$ from to a ruled $K$-variety.
\end{remark}

Combining Theorem \ref{thm:gen-matsusaka} with the proof of \cite[Lemma 1.3]{CS20}, we obtain the next corollary.
See \cite[Theorem IV.1.8]{Kollar1996} for a related result that works over arbitrary uncountable fields.

\begin{definition}
  \label{def:universal-domain}
  A \emph{universal domain} is an algebraically closed field of infinite transcendence degree over its prime subfield.
\end{definition}

\begin{example}
  The field of complex numbers $\C$ is a universal domain.
\end{example}

\begin{corollary}
  \label{cor:very-general}
  Let $k$ be a universal domain.
  Let $S$ be a variety over $k$ and let $f : X\to S$ be a flat, finite-type morphism whose very general fiber is integral.
  Let $s, t\in S$ be closed points with $t$ very general.
  Let $X_s$ and $X_t$ be the corresponding fibers of $f$, and let $X'_s\subseteq X_s$ be an irreducible component at the generic point of which $X_s$ is reduced. If $X_t$ is ruled, then so is $X_s'$. In general,
  \begin{equation*}
    \nr(X'_s)\le \nr(X_t).
  \end{equation*}
\end{corollary}

\begin{proof}
  Let $\eta$ be the generic point of $S$.
  By \cite[Proof of Lemma 2.1]{Vial2013}, there exist an isomorphism of fields $\alpha : k\xrightarrow\sim \overline{k(S)}$ and an isomorphism of schemes $\beta : X_{\bar \eta}\xrightarrow\sim X_t$ such that the diagram
  \begin{equation*}
    \begin{tikzcd}
      X_{\bar \eta} \ar[r, "\beta", "\sim"'] \ar[d,"f"] &
      X_t \ar[d, "f"] \\
      \Spec\overline{k(S)} \ar[r, "\Spec\alpha", "\sim"'] &
      \Spec k
    \end{tikzcd}
  \end{equation*}
  commutes.
  Hence the result follows immediately from Theorem \ref{thm:gen-matsusaka}.
\end{proof}

\section{Proof of the main theorem}

\label{sec:proof of the main theorem}

In this section, we prove our main strongest result about complete intersections, namely Theorem \ref{thm:nonruledness-precise}, and use it to establish Theorem \ref{thm:invariants-intro} from the introduction.
At the heart of the proof of Theorem \ref{thm:nonruledness-precise} is a nonruledness result about certain flat limits of complete intersections, namely Theorem \ref{thm:special-fiber}.
We prove that result by combining Proposition \ref{prop:obstruction-to-ruledness} with the main results of \cite{Braune19b}.

We begin by recalling the construction of cyclic covers.

\begin{definition}
  \label{def:cyclic-covers}
  Let $X$ be a scheme and let $L$ be an invertible $\OO_X$-module.
  Let
  \begin{equation*}
    V := \Spec_X \Sym (L^\vee) \xrightarrow \pi X
  \end{equation*}
  be the vector bundle corresponding to $L$.
  Let $\tau\in \Gamma(V,\pi^* L)$ be the \emph{tautological section}, which corresponds to the identity of $V$ under the natural bijections
  \begin{equation*}
    \Hom_X(V, V) \cong \Hom_{\OO_V\text{-alg.}}(\pi^* \Sym (L^\vee), \OO_V) \cong \Gamma(V,\pi^* L).
  \end{equation*}
  Let $m$ be a positive integer and let $s\in\Gamma(X,L^{\otimes m})$ be a section.
  The \emph{degree-$m$ cyclic cover} of $X$ determined by $s$ is the closed subscheme
  \begin{equation*}
    X[\displayroot m s] \subseteq V
  \end{equation*}
  defined by the equation $\tau^{\otimes m} = \pi^*s$, equipped with its natural morphism to $X$.
\end{definition}

The following example describes cyclic covers locally.

\begin{example}
  Suppose that $L$ is free with generator $v\in \Gamma(X, L)$. 
  Thus $s = f\cdot v$ for some $f\in \Gamma(X, \OO_X)$.
  There exists a unique isomorphism of $X$-schemes $\varphi : X\times \A^1 \xrightarrow\sim V$ such that
  \begin{equation*}
    \varphi^* \tau = t \cdot \pr_1^* v
  \end{equation*}
  where $t$ is the coordinate on $\A^1$.
  The cyclic cover $X[\pthroot s]$ corresponds under this isomorphism to the subscheme of $X\times \A^1$ defined by $t^m = f$.
\end{example}

A key ingredient in Kollár's irrationality result for hypersurfaces  is Mori's degeneration of hypersurfaces of even degree $2a$ to double covers of hypersurfaces of degree $a$; see \cite[Example 4.3]{Mori1975}.
The next proposition describes the variant of Mori's degeneration that we use to establish Theorem \ref{thm:nonruledness-precise}.

\begin{setup}
  \label{setup:cover}
  Let $d_1,\dotsc, d_c$ be positive integers.
  Fix $m$ with $1\le m \le c$, and suppose that $p\ge 1$ is a common divisor of $d_1,\dotsc, d_m$.
  Let $k$ be a field, and let $X\subset \PP^N_k$ be a positive-dimensional complete intersection of multi-degree
  \begin{equation*}
    (d_1/p, \dotsc, d_m/p, d_{m+1}, \dotsc, d_c).
  \end{equation*}
  Let $f_1,\dotsc, f_m\in k[x_0,\dotsc, x_N]$ be homogeneous polynomials of respective degrees $d_1,\dotsc, d_m$.
  Finally, let $Y$ denote the fiber product
  \begin{equation*}
    Y := X[\pthroot{f_1}] \times_X \dotsb \times_X X[\pthroot{f_m}].
  \end{equation*}
\end{setup}

\begin{proposition}
  \label{prop:mori}
  Assume Setup \ref{setup:cover}.
  Let $R$ be a DVR with fraction field $K$ and residue field $k$.
  There exists a flat proper morphism
  \begin{equation*}
    \pi : Z\to \Spec R
  \end{equation*}
  whose generic fiber is isomorphic as a $K$-scheme to a complete intersection of multi-degree $(d_1,\dotsc, d_c)$ in $\PP^N_K$, and
  whose special fiber is isomorphic as a $k$-scheme to $Y$.
\end{proposition}

\begin{proof}
  Follows immediately from \cite[Lemma 3.42]{Braune19b}.
\end{proof}

The next result establishes, under certain conditions, a lower bound for the degree of separable nonruledness of the special fiber of the degeneration of Proposition \ref{prop:mori}. It combines Proposition \ref{prop:obstruction-to-ruledness} with the main results of \cite{Braune19b}.

\begin{theorem}
  \label{thm:special-fiber}
  Assume Setup \ref{setup:cover}.
  Suppose
  \begin{itemize}
    \item the field $k$ is algebraically closed of characteristic $p$,
    \item the complete intersection $X \subseteq \PP^N_k$ is smooth over $k$, 
    \item $f_i \in k[x_0, $ $\dotsc, x_N]$ is a general member of the vector space of homogeneous polynomials of degree $d_i$ for all $i=1,\dotsc, m$.
  \end{itemize}
  Suppose furthermore
  \begin{equation}
    \label{eqn:dimensions}
    m \le \min\left(
      N-c-1,
      \tfrac 1 2 (N-c+3),
      \tfrac 1 2 (N-c-m-1)(N-c-m-1\pm 1)
    \right),
  \end{equation}
  where the symbol ``$\pm$'' should be read as ``plus'' if $p > 2$ and as ``minus'' if $p = 2$.
  Finally, suppose
  \begin{equation*}
    \delta := \frac {p+1} p \sum_{i\le m} d_i + \sum_{i>m} d_i - N -1 > 0.
  \end{equation*}
  Then $Y$ is normal and irreducible. Furthermore,
  \begin{equation*}
    \snr(Y) > \lfloor (\delta+1)/2\rfloor.
  \end{equation*}
\end{theorem}

\begin{proof}
  Let $\Frob_X : X \to X$ be the absolute Frobenius morphism, which acts on the topological space of $X$ as the identity and whose comorphism $\Frob_X^\# : \OO_X\to \OO_X$ sends $f\mapsto f^p$.
  
  Let $E := \oplus_{i=1}^m \OO_X(d_i/p)$.
  Then $\Frob_X^* E = \oplus_{i=1}^m \OO_X(d_i)$.
  Let $s := (f_1,\dotsc, f_m)|_X \in \Gamma(X, \Frob_X^* E)$.
  Applying \cite[Construction 3.11]{Braune19b} to $s$, we obtain a morphism of schemes $\mu : X[\pthroot s]\to X$. This construction recovers $Y$ in the sense that there exists a canonical isomorphism of $X$-schemes
  \begin{equation*}
    Y \cong X[\pthroot s].
  \end{equation*}

  The Frobenius pullback $\Frob_X^* E$ carries a canonical connection \cite[Definition 3.9]{Braune19b}.
  Given a nonnegative integer $i$, one can use this connection to define \emph{$i$th critical locus of $s$}, which is a closed subscheme $\Sigma^i(s)\subseteq X$ \cite[Definition 1.30]{Braune19b}.
  Proceding by analogy with \cite{Porteous71}, given nonnegative integers $i$ and $j$, one can similarly define the \emph{locus of second-order singularities with symbol $(i,j)$ of $s$}, which is a locally closed subscheme $\Sigma^{i,j}(s)\subseteq X$ \cite[Definition 1.33]{Braune19b}.

  Let $W$ be the image of the $k$-linear restriction map
  \begin{equation*}
    \Gamma(\PP^N_k, \oplus_{i=1}^m \OO(d_i)) \rightarrow \Gamma(X, \Frob_X^* E).
  \end{equation*}
  The section $s\in \Gamma(X, \Frob_X^* E)$ is a general element of $W$.
  We have $d_i\ge 2$ for all $i=1,\dotsc, m$, so the natural $k$-linear map
  \begin{equation*}
    W\to \Frob_X^* E \otimes \OO_X/\mathfrak m_x^3
  \end{equation*}
  is surjective for all closed points $x\in X$.
  Applying \cite[Corollaries 2.36 and 2.59]{Braune19b}, we obtain formulas for the codimensions of $\Sigma^i(s)$ and $\Sigma^{i,j}(s)$ in $X$.
  
  Comparing these formulas with (\ref{eqn:dimensions}), we conclude
  \begin{enumerate}
    \item $\codim_X \Sigma^1(s) \ge 2$,
    \item $\Sigma^2(s)$ is empty, and
    \item $\Sigma^{1, N-c-m-1}(s)$ is empty.
  \end{enumerate}
  Thus the hypotheses \cite[Propositions 3.20 and 3.29]{Braune19b} are satisfied.
  According to these results, $Y= X[\pthroot s]$ is integral and normal. Moreover, letting $Q := \omega_X \otimes \det(E)^{\otimes p}$, there exist a proper birational map $\rho : B\to Y$ and an injection
  \begin{equation*}
    \rho^*\mu^* Q\hookrightarrow \Omega_B^{N-c-m}/(\Omega_B^{N-c-m})_\tors.
  \end{equation*}

  By the adjunction formula, there exists an isomoprhism of $\OO_X$-modules $\omega_X \otimes \det(E)^{\otimes p} \cong \OO_X(\delta)$.
  It follows that $Q$ separates $\delta+1$ general points (Remark \ref{rmk:sep-pts}).
  The same is true of $\rho^*\mu^* Q$, because $\mu$ is purely inseparable (hence a homeomorphism) and $\rho$ is birational. Applying Proposition \ref{prop:obstruction-to-ruledness}, we conclude
  \begin{equation*}
    \snr(Y) = \snr(B) > \lfloor (\delta+1)/2\rfloor. \qedhere
  \end{equation*}
\end{proof}

\begin{remark}
  \label{rmk:finite-set}
  It can be shown that condition (\ref{eqn:dimensions}) is equivalent to the following, which can be easier to check:
  \begin{equation}
    m \le \tfrac 1 2 (N-c + 3) \qquad\text{and}\qquad (m, N-c) \not\in \U_\pm
  \end{equation}
  where 
  \begin{align*}
    \U_+ & := \{ (1, 1), (2, 1), (1, 2), (2, 2), (2, 3), (3, 3), (2, 4), (3, 4), (3, 5), (4, 5), (4, 6), \\
         & \qquad (4, 7), (5, 7), (5, 8), (6, 9), (7, 11) \}                                            \\
    \U_- & := \U_+ \cup \{ (1, 3), (2, 5), (3, 6), (4, 8), (5, 9), (6, 10), (7, 12), (8, 13) \}.
  \end{align*}
  As before, the symbol ``$\pm$'' should be read as ``plus'' if $p > 2$ and as ``minus'' if $p = 2$.
  Thus (\ref{eqn:dimensions}) holds if $m \le \tfrac 1 2 (N-c + 3)$ and $N-c \ge 14$.
\end{remark}

The next theorem is the strongest irrationality result for complete intersections that is proved in this paper.
As we shall see, it has Theorem \ref{thm:invariants-intro} from the introduction as a corollary.
Another result with a simple statement that follows from Theorem \ref{thm:nonruledness-precise} is Corollary \ref{cor:quadrics-intro} below.

\begin{theorem}
  \label{thm:nonruledness-precise}
  Let $p$ be a prime number.
  Let $k$ be a universal domain (Definition \ref{def:universal-domain}) of characteristic zero or $p$.
  Let $X\subseteq \PP^N_k$ be a very general complete intersection of positive dimension and multi-degree $(d_1,\dotsc, d_c)$ over $k$.
  Let $m$ be an integer such that
  \begin{itemize}
    \item $1\le m\le \min(c, \tfrac 1 2 (N-c+3))$,
    \item $(m, N-c) \not\in \U_\pm$ (the finite set of Remark \ref{rmk:finite-set}), and
    \item $d_1,\dotsc, d_m \ge p$,
  \end{itemize}
  For $i=1,\dotsc, m$, let $r_i \in \{0,1, \dotsc, p-1\}$ denote the residue of $d_i$ modulo $p$.
  Suppose that
  \begin{equation*}
    \delta := \dfrac{p+1}{p}\sum_{i\le m}(d_i - r_i) + \sum_{i>m} d_i - N -1 > 0.
  \end{equation*}
  Then $\nr(X) \ge \min(\lfloor (\delta+3)/2 \rfloor,p)$.
\end{theorem}

\begin{proof}
  By Corollary \ref{cor:very-general} applied to the family of all complete intersections of multi-degree $(d_1,\dotsc, d_c)$ in $\PP_k^N$, it suffices to show that there exists an integral complete intersection $X'\subseteq \PP^N_k$ of multi-degree
  \begin{equation*}
    (d_1-r_1,\dotsc, d_m-r_m, d_{m+1}, \dotsc, d_c)
  \end{equation*}
  for which the conclusion holds.
  Indeed, multiplying the first $m$ equations defining $X'$ by general homogeneous polynomials of degrees $r_1,\dotsc, r_m$, we obtain a reduced complete intersection in $\PP^N_k$ of multi-degree $(d_1,\dotsc, d_c)$ that contains $X'$ as a reduced irreducible component.
  In particular, we may assume $d_1,\dotsc, d_m$ are divisible by $p$.

  Let $R$ be an excellent DVR whose fraction field embeds in $k$ and whose residue field is algebraically closed of characteristic $p$.
  If $\charac k = p$, we can take $R$ to be the local ring of the affine line over $\overline \FF_p$ at the origin.
  If $\charac k = 0$, we can take $R := W\cap L$, where $W$ is the ring of Witt vectors of $\overline \FF_p$ and $L\subseteq \Frac W$ is a countable subfield containing lifts of all elements of $\overline \FF_p$.

  By Proposition \ref{prop:mori}, Theorem \ref{thm:special-fiber} and Remark \ref{rmk:deg-insep}, there exists a flat proper morphism $\pi : Z\to \Spec R$ whose generic fiber is isomorphic to a complete intersection of multi-degree $(d_1,\dotsc, d_c)$ in $\PP^N_{\Frac R}$, and whose special fiber is a variety for which the conclusion holds.
  We can apply Theorem \ref{thm:gen-matsusaka} to the morphism $\pi : Z\to \Spec R$ because its special fiber is geometrically integral, and therefore so is its generic fiber.
  Thus the conclusion holds for the geometric generic fiber and its base change to $k$, see Remark \ref{rmk:ruled-geometric}.
\end{proof}

\begin{corollary}
  \label{cor:quadrics-intro}
  Let $X \subseteq \PP^N_\C$ be a very general complete intersection of $c$ quadrics.
  If $c \ge N/3 + 1$ and $(c, N-c) \not\in \U_2$ (the finite set of Remark \ref{rmk:finite-set}), then $X$ is not ruled (and therefore not rational).
\end{corollary}

\begin{proof}
  Follows from Theorem \ref{thm:nonruledness-precise} with $p=2$ and $m=\lfloor \tfrac 1 2 (N-c+3) \rfloor$.
  The hypotheses of the theorem are satisfied because $c \ge N/3 + 1$ guarantees $m\le c$ and $\delta := 3m + 2(c-m) - N - 1 > 0$.
\end{proof}

\begin{proof}[Proof of Theorem \ref{thm:invariants-intro}]
  If $\sum d_i \ge N + p - 1$, then $\snr(X) > p - 1$ by Corollary \ref{cor:no-degeneration}, so $\nr(X) \ge p$ by Remark \ref{rmk:deg-insep}.
  Thus we may assume
  \begin{equation*}
    \frac p {p+1} (N+2p-2) + c(p-1) \le \sum_{i=1}^c d_i < N + p - 1.
  \end{equation*}
  Let's verify the hypotheses of Theorem \ref{thm:nonruledness-precise} with $m=c$.

  Multiplying the displayed inequalities by $p+1$, we obtain
  \begin{equation*}
    p(N+2p-2) + c(p^2-1) < (p+1)N + p^2 - 1.
  \end{equation*}
  Subtracting $pN + p^2 - 1$ from both sides yields
  $
    (p-1)^2 + c(p^2-1) < N.
  $
  Hence $1 + 3c < N$, that is, $2 + 3c \le N$.
  Using this inequality, it is easy to check that $c\le \tfrac 1 2 (N - c + 3)$, $c\le N-c-1$ and $c\le (N-2c-1)(N-2c-1 \pm 1)$.

  Let
  \begin{equation*}
    \delta := \dfrac{p+1} p \sum_{i=1}^c (d_i - r_i) - N - 1.
  \end{equation*}
  We have
  \begin{align*}
    \sum_{i=1}^c (d_i - r_i) \ge \sum_{i=1}^c d_i - c(p-1) \ge \dfrac p {p+1} (N+2p-2),
  \end{align*}
  so $\delta \ge 2p-3$ and $\lfloor(\delta + 3)/2\rfloor \ge p-1$.
  Thus Theorem \ref{thm:nonruledness-precise} applies, and the result follows.
\end{proof}

\bibliographystyle{amsplain}
\bibliography{master-bib}

\end{document}